\documentclass{amsart}

\usepackage[all]{xy}
\usepackage{hyperref}
\usepackage{amssymb}
\usepackage{mathdots}
\usepackage{mathtools}

\usepackage{tikz,tkz-euclide}

\hypersetup{
    colorlinks,    citecolor=blue,    filecolor=blue,    linkcolor=blue,    urlcolor=blue
}

\usepackage{booktabs}
\usepackage{mathrsfs}
\usepackage{todonotes}
\usepackage{bbm}
\usepackage[lite,bibtex-style]{amsrefs}
\usepackage{longtable}

\usepackage{MnSymbol}

\newtheorem{introthm}{Theorem}

\newtheorem{introcor}[introthm]{Corollary}

\newtheorem{theorem}{Theorem}[section]
\newtheorem{lemma}[theorem]{Lemma}
\newtheorem{proposition}[theorem]{Proposition}

\newtheorem{conjecture}[theorem]{Conjecture}
\theoremstyle{definition}

\newtheorem{definition}[theorem]{Definition}
\newtheorem{example}[theorem]{Example}

\newtheorem{question}[theorem]{Question}

\newtheorem{remark}[theorem]{Remark}
\theoremstyle{remark}

\numberwithin{equation}{section}

\def\qq{{\mathbb Q}}
\def\pp{{\mathbb P}}

\def\ls{{\mathcal L}}

\newcommand{\Eff}{\operatorname{Eff}}

\newcommand{\lw}{\operatorname{lw}}
\newcommand{\vol}{\operatorname{vol}}

\newcommand{\Pic}{\operatorname{Pic}}
\newcommand{\res}{\operatorname{res}}
\newcommand{\ord}{\operatorname{ord}}
\newcommand{\mult}{\operatorname{mult}}

\title{On intrinsic negative curves}
\author[A.~Laface]{Antonio Laface}
\address{
Departamento de Matem\'atica,
Universidad de Concepci\'on,
Casilla 160-C,
Concepci\'on, Chile}
\email{alaface@udec.cl}

\author[L.~Ugaglia]{Luca Ugaglia}
\address{
Dipartimento di Matematica e Informatica,
Universit\`a degli studi di Palermo,
Via Archirafi 34,
90123 Palermo, Italy}
\email{luca.ugaglia@unipa.it}

\thanks{Both authors have been 
partially supported by Proyecto
FONDECYT Regular n. 1190777}

\date{\today}

\begin{document}

\begin{abstract}
Let $\mathbb K$ be an algebraically
closed field of characteristic $0$.
A curve of $(\mathbb K^*)^2$
arising from a Laurent polynomial in
two variables is {\em intrinsic negative} if its tropical compactification
has negative self-intersection.
The aim of this note is to start a 
systematic study of these curves
and to relate them with the problem of computing Seshadri constants of toric surfaces.
\end{abstract}

\maketitle

\section*{Introduction}
Following the work of Gonz\'alez Anaya, Gonz\'alez, Karu
\cite{ggk}, Kurano~\cite{kur} and Kurano Matsuoka~\cite{km}, we define a class
of curves on the blowing-up of toric
surfaces at a general point.
Let $f$ be a Laurent polynomial
in two variables and let $\Gamma\subseteq(\mathbb K^*)^2$
be its zero locus. The normal fan to 
the Newton polygon 
$\Delta$
of $f$ defines a
toric variety $\mathbb P$ such that
the compactification of $\Gamma$ is
contained in the smooth locus of $\mathbb P$.
Such a compactification is called
{\em tropical}, see~\cite{te}.
Denote by $X := {\rm Bl}_e\mathbb P$ the blowing-up of $\mathbb P$
at the image $e$ of $(1,1)$ and let
$C$ be the strict transform of the
compactified curve.
We say that $C$ is an 
{\em intrinsic negative curve} (resp. {\em non positive})
if $C^2<0$ (resp. $C^2 \leq 0$), cfr.~\cite{kur}*{Definition~3.1}.
Our first result is the construction
of infinite families of intrinsic
non-positive curves.
In the following table ${\rm lw}(\Delta)$
is the lattice width of $\Delta$,
defined in Section~\ref{s:int-cur},
while $g(C)$ is the genus of the curve $C$.

\begin{introthm}
\label{prop:inf}
There exist infinite families of non-positive intrinsic curves, whose
Newton polygons are listed in the following table

\begin{center}
\begin{longtable}{|c|c|c|c|c|}
\hline
& vertices of $\Delta$ & $\lw(\Delta)$ & $C^2$ & $g(C)$ \\
\hline
&&&&\\
(i)
&
$
\left[
\begin{smallmatrix}
0 & m & 1\\
0 & 1 & m
\end{smallmatrix}
\right]
$ 
& 
$m\geq 2$
&
$-1$
& 
$0$ 
\cr
&&&&\\
\hline
&&&&\\
(ii)
&
$
\left[
\begin{smallmatrix}
0 & m-3 & m & m-1 & m-2\\
0 & 0 & 1 & m & m-1
\end{smallmatrix}
\right]
$  
& 
$m\geq 4$
&
$-1$ 
& 
$0$ 
\\
&&&&\\
\hline
&&&&\\
(iii)
&
$
\left[
\begin{smallmatrix}
0 & 0 & 2 & m-4 & m-1 & m & m-1\\
0 & 1 & m & m & m-1 & m-2 & m -3
\end{smallmatrix}
\right]
$  
& 
$m = 2k \geq 8$
&
$-2$ 
& 
$0$ 
\\
&&&&\\
\hline
&&&&\\
(iv)
&
$
\left[
\begin{smallmatrix}
0 & m-2 & m & m-1 & m-2\\
0 & 0 & 1 & m & m-1
\end{smallmatrix}
\right]
$  
& 
$m\geq 4$
&
$0$ 
& 
$0$ 
\\
&&&&\\
\hline
&&&&\\
(v)
&
$
\left[
\begin{smallmatrix}
0 & m-4 & m & m-2 & m-3\\
0 & 0 & 1 & m & m-1
\end{smallmatrix}
\right]
$  
& 
$m = 2k + 4 \geq 6$
&
$0$ 
& 
$1$ 
\\
&&&&\\
\hline
\end{longtable}
\end{center}
\end{introthm}

Before stating the next result 
we recall that given a projective variety $\pp$, an ample class $H$ and a point $x\in \pp$, the {\em Seshadri constant} of $H$ at $x$
can be defined as
\[
 \varepsilon(H,x)
 := \inf_{x\in C}
 \frac{H\cdot C}{\mult_x(C)}
\]
where the infimum is taken over all irreducible curves through $x$.
The problem of finding Seshadri constants
of algebraic surfaces have been widely
studied (see for instance
\cites{bau,el,fss,nak} and the references therein).
When $\pp$ is a toric surface there are 
three possibilities for $x\in \pp$: either the point is torus-invariant, or it lies 
on a torus-invariant curve, or it is general.
In the first two cases, since the blowing-up ${\rm Bl}_x\mathbb P$ admits the action of a torus of dimension two and one respectively, it is possible to describe the effective cone (see~\cite{cls,ow} and~\cite{adhl}*{\S 5.4} for a 
description of the Cox ring), and hence to compute the Seshadri constant (see \cite{primer}*{\S4} and~\cite{ito}*{\S3.2}). 
Concerning a general point, in~\cite{ito}*{Thm.~1.3}) a lower bound for the Seshadri constant is given. 

In this note we focus on the case of a general point.  
In particular we prove some relations between the geometry of a lattice polygon $\Delta$ and the Seshadri constant $\varepsilon(H_\Delta,e)$, where $(\pp_\Delta,H_\Delta)$ is the {\em toric pair} defined by $\Delta$
(see \S~\ref{s:int-cur} for the  definition)
and $e\in \pp_\Delta$ is a general point.
The recent interest in these Seshadri 
constants and more generally in the 
Cox ring of blow-ups of toric varieties 
at a general point has been motivated
by the work of Castravet and Tevelev
\cite{ct} where the authors prove that the
finite generation of the the Cox ring of 
$\overline M_{0,n}$ implies that of certain
blow-ups of toric varieties at a general point.

In order to state our result, given a non-negative integer $m$
denote by $\ls_\Delta(m)$ the linear
system of Laurent polynomials whose
exponents are integer points of
$\Delta$ and such that all the 
partial derivatives up to order
$m-1$ vanish at $(1,1)$.
If we denote by $\vol(\Delta)$
the normalized volume of $\Delta$
(that is twice its euclidean area),
we have the following
(the first inequality is well known \cite{ai}*{Thm.~0.1} and \cite{ito}*{}, but we state it anyway for the sake of completeness).

\begin{introthm}
\label{prop:sesh-m}
Let $\Delta\subseteq\qq^2$ be a lattice polygon, let $(\pp_\Delta,H_\Delta)$ be the corresponding toric pair and let $\varepsilon := \varepsilon(H_\Delta,e)$ be the Seshadri constant at $e\in\mathbb P_\Delta$.
Then the following hold.
\begin{enumerate}
    \item 
    \label{pro2-i}
    $\varepsilon\leq \lw(\Delta)$.
    \item
    \label{pro2-ii}
    If $\vol(\Delta) > \lw(\Delta)^2$ then $\varepsilon \in \qq$.
    \item
    \label{pro2-iii}
    If there exists
    $m\in\mathbb N$ such that
    $\vol(\Delta) \leq m^2$   and $\ls_\Delta(m) \neq\emptyset$, then $\varepsilon\in\qq$ and  
    $\varepsilon\leq\vol(\Delta)/m$.
    \item
    \label{pro2-iv}
    If moreover $\ls_\Delta(m)$ contains an irreducible curve, then $\varepsilon = \vol(\Delta)/m$
\end{enumerate}
\end{introthm}

We remark that Theorem~\ref{prop:sesh-m}
provides in some cases (like 
e.g.~\cite{hkw}*{Example 5.7})
an alternative proof for the rationality
of the Seshadri constant of a toric
surface at a general point.
Moreover it allows to compute the exact value of the Seshadri constant $\varepsilon(H_\Delta,e)$ when $(X_\Delta,H_\Delta)$ is the toric pair associated to the Newton polygon of an intrinsic non-positive curve. 

\begin{introcor}
\label{cor:inf}
Let $f\in
\mathbb K[u^{\pm 1},v^{\pm 1}]$ be a 
Laurent polynomial with Newton polygon
$\Delta$ and multiplicity $m$ at
$(1,1)$, such that the corresponding
intrinsic curve $C\subseteq 
X_\Delta$ is non-positive, i.e. 
$C^2\leq 0$. Then the Seshadri constant
of the ample divisor $H_\Delta$ of the 
toric surface $\mathbb P_\Delta$ at a
general point $e\in\mathbb P_\Delta$ is
\[
 \varepsilon
 =
 \frac{\vol(\Delta)}{m}.
\]
In particular the polygons of the infinite families appearing in Theorem~\ref{prop:inf} have Seshadri constant $\varepsilon = \vol(\Delta)/\lw(\Delta)\in \qq$.
\end{introcor}

 In order to prove the last statement we are going to apply Theorem~\ref{prop:sesh-m}{(iv)}, showing that in each case there exists an irreducible curve in $\ls_\Delta(m)$, where $m = \lw(\Delta)$ and $\vol(\Delta) \leq m^2$. We remark that for the triangles of type (i) in Theorem~\ref{prop:inf}, the upper bound of Theorem~\ref{prop:sesh-m}{(iii)} coincides with the lower bound given by~\cite{ito}*{\S3.2}, so that it is also possible to deduce the exact value of the Seshadri constant without producing the irreducible curve, but for all the other families of polygons appearing in Theorem~\ref{prop:inf}, the two bounds are different (see also Remark~\ref{rem:ito}).

The paper is structured as follows. 
In \S~\ref{s:int-cur}, after recalling some definitions and results about toric varieties and lattice polytopes we introduce intrinsic curves and we prove some preliminaries result. In \S~\ref{s:inf} we consider infinite families of intrinsic curves:  we first prove Theorem~\ref{prop:inf}, and then we construct an infinite family of intrinsic negative curves on a given toric surface
(Example~\ref{ex:inf}). Finally, 
\S~\ref{s:sesh} is devoted to Seshadri constants on toric surfaces. We first prove Theorem~\ref{prop:sesh-m} 
and Corollary~\ref{cor:inf}, and then we discuss some possible applications to the study of the blowing up of weighted projective planes at a general point.

\section{Intrinsic curves}
\label{s:int-cur}
Let us first recall some definitions and set some notations we are going to use throughout this note. 

Let $\Delta\subseteq \qq^n$ be a {\em lattice polytope}  
i.e. a polytope whose vertices have integer coordinates.
We recall that given a non zero primitive vector $v\in \mathbb Z^n$ the {\em lattice width of}
$\Delta$ {\em in the direction} $v$ is $\lw_v(\Delta) := \max(\Delta,v) - \min(\Delta,v)$ and the {\em lattice width
of} $\Delta$ is $\lw(\Delta) := \min\{\lw_v(\Delta)\, :\,  v\in \mathbb Z^n\}$,
see~\cite{lu}*{Def. 1.8}.

Given a lattice polytope $\Delta\subseteq \qq^n$ we can define a pair $(\mathbb P_\Delta,H_\Delta)$ 
consisting of a toric variety 
$\mathbb P_\Delta$ together with a
very ample divisor $H_\Delta$.
The toric variety is the normalization
of the closure of the image of the 
following monomial morphism
\[
 g_\Delta\colon (\mathbb K^*)^n
 \to\pp^{|\Delta\cap \mathbb Z^n|-1},
 \qquad
 u \mapsto [u^w : w\in \Delta\cap \mathbb Z^n],
\]
where $u = (u_1,\dots,u_n)\in(\mathbb K^*)^n$.
It is possible to show that the action 
of the torus $(\mathbb K^*)^n$ on itself
extends to an action on $\mathbb P_\Delta$
and that the subset of prime torus-invariant
divisors is finite and in bijection
with the set of facets of $\Delta$.
Let $D_1,\dots,D_r$ be such divisors and 
let $v_1,\dots,v_r$ be the inward normal 
vectors to the facets of $\Delta$.
Each $v_i$ defines a linear form 
$\mathbb Q^n\to\mathbb Q$ by $w\mapsto
w\cdot v_i$ and the very ample divisor is 
~\cite{hkw}*{Prop.~3.1}:
\begin{equation}
\label{very ample}
 H_\Delta := 
 -\sum_{i=1}^r\min_{w\in \Delta}
 \{w\cdot v_i\} D_i.
\end{equation}
On the other hand, any divisor $D$ on $\mathbb P_\Delta$ is equivalent to a combination $\sum_i a_iD_i$, so that we can
associate to it the {\em Riemann-Roch polytope} 
\[
 \Delta_D 
 := 
 \{w\in \mathbb Q^n\, :\, w\cdot v_i \geq -a_i,\ \forall i = 1,\dots,r\}
\]
(see~\cite{cls}*{\S 4.3}). We remark that if $D$ is very ample then the toric variety associated to $\Delta_D$ coincides with $\mathbb P_\Delta$.
We recall that if $\Delta$ is a
very ample polytope~\cite{cls}*{Def. 2.2.17}
then the closure of the image of $g_\Delta$
is a normal variety by
\cite{cls}*{Thm. 2.3.1}
and thus it coincides with 
$\mathbb P_\Delta$.
Moreover, by~\cite{cls}*{Cor. 2.2.19}, 
$\Delta$ is very ample if $n = 2$.

From now on we restrict to the case $n=2$, i.e. $\Delta\subseteq\mathbb Q^2$ is a lattice polygon, so that $\mathbb P_\Delta$ is a normal toric surface. We will denote by $e\in\pp_\Delta$ the image via $g_\Delta$ of the neutral element of the torus, by $\pi\colon X_\Delta\to\pp_\Delta$ the blowing up of $\mathbb P_\Delta$ at $e$ and by $E$ the exceptional divisor. Given an $m\in\mathbb N$ we will denote by $\ls_{\Delta}(m)$  the sublinear system of $|H_\Delta|$ consisting of
sections having multiplicity at least $m$ at $e$.

\begin{definition}
Let $f\in\mathbb K[u^{\pm 1},v^{\pm 1}]$ be an irreducible Laurent polynomial, let $\Delta\subseteq \mathbb Q^2$ be the Newton polygon of $f$, i.e. 
the convex hull of its exponents, and let
$\Gamma\subseteq \mathbb P_\Delta$ be the closure of $V(f)\subseteq (\mathbb K^*)^2$. We say that  
the strict transform $C\subseteq X_\Delta$ of $\Gamma$
is the {\em intrinsic curve} defined by $f$ and that $C$ is:
\begin{itemize}
\item
an {\em intrinsic negative}
(resp. {\em non-positive}) curve if $C^2 < 0 $
(resp. $C^2\leq 0$);
\item
an {\em intrinsic $(-n)$-curve} if 
$C^2 = -n < 0$ and $p_a(C) = 0$;
\item
{\em expected} in $X_{\Delta'}$,
with $\Delta\subseteq\Delta'$ if $|\Delta'\cap \mathbb Z^2| >
{{m+1}\choose{2}}$.
\end{itemize}
\end{definition}

 In what follows we will often use the notation $\mathbb P,\ X$ and $H$, omitting the subscript when it is clear from the context.

We remark that, with the notation above, $\Gamma\subseteq \mathbb P$ 
is an  element of the very ample
linear series $|H|$ and  
$\Gamma\in\ls_\Delta(m)$
if $f$ has multiplicity at least $m$
at $(1,1)$, that is all the partial
derivatives of $f$ up to order $m-1$ vanish at $(1,1)$. Moreover if the
multiplicity is $m$ then the strict transform
$C\subseteq X$ of $\Gamma$ is a Cartier divisor such that 
\begin{equation}
\label{equ:num}
 C^2 = \vol(\Delta) - m^2
 \qquad
 p_a(C) = \frac12
 \left(\vol(\Delta) 
 - |\partial\Delta\cap\mathbb Z^2|
 + m - m^2\right) +1,
\end{equation}
see for instance~\cite{cltu}*{\S 4}.
By abuse of notation we will sometimes refer to the Newton polygon $\Delta$ as the {\em Newton polygon of} $C$, and we will simply say that 
$C$ is {\em expected} if it is expected in 
$X_\Delta$, that is the linear system
$\mathcal L_{\Delta}(m)$ has a non-negative
expected dimension.

Our first result is about the 
characterization of Newton polygons
of expected non-positive curves.

\begin{proposition}
\label{prop:exp}
Let $C$ be an intrinsic non-positive expected curve with Newton polygon $\Delta$ and multiplicity $m$ at $e$.
Then one of the following holds:
\begin{itemize}
    \item 
    $\vol(\Delta) = m^2$ and $|\partial\Delta\cap \mathbb Z^2| = m$;
    \item 
    $\vol(\Delta) = m^2$ and $|\partial\Delta\cap \mathbb Z^2| = m + 2$;
    \item 
    $\vol(\Delta) = m^2 - 1$ and $|\partial\Delta\cap \mathbb Z^2| = m + 1$.
\end{itemize}
In particular $C^2\in\{-1,0\}$ and 
$p_a(C)\in\{0,1\}$.
\end{proposition}
\begin{proof}
Let us denote by $b := |\partial\Delta\cap\mathbb Z^2|$
the number of boundary lattice points
of $\Delta$ and by $i := |\Delta\cap\mathbb Z^2|-b$
the number of interior lattice points.
Recall that by Pick's formula
$\vol(\Delta) = 2i+b-2$.
Since $C$ is expected and non-positive, we have 
$|\Delta\cap\mathbb Z^2|\geq \binom{m+1}{2}+1$
and $\vol(\Delta) \leq m^2$. 
By~\eqref{equ:num},
the non-negativity of the arithmetic 
genus of $C$ gives 
$\frac{1}{2}(\vol(\Delta)-b+m-m^2)+1\geq 0$.
The three inequalities in terms of 
$i$ and $b$ are
$2i+2b\geq m(m+1)$, $2i+b-2\leq m^2$,
$2i-m^2+m\geq 0$. From these one deduces
that one of the following holds:
\[
\begin{array}{lll}
\left\{
\begin{array}{l}
 b = m\\
 i = \frac{m^2-m}2 + 1
\end{array}
\right.
&
\left\{
\begin{array}{l}
 b = m + 2\\
 i = \frac{m^2-m}2.
\end{array}
\right.
&
\left\{
\begin{array}{l}
 b = m+1\\
 i = \frac{m^2-m}2
\end{array}
\right.
\end{array}
\]
Since $C^2 = \vol(\Delta) - m^2$ and 
$C\cdot K = m-b$, in the first two cases we have $C^2 = 0$ and $p_a(C) = 1$ and $0$ respectively, while in the last one $C^2 = -1$ and $p_a(C) = 0$.
\end{proof}

\begin{proposition}
\label{pro:class}
All the non-equivalent polygons for 
intrinsic non-positive curves of
size=2
\begin{center}
\begin{tabular}{c|l}
$m$ & $\Delta$\\
\hline
\ \\[-5pt]
$2$ &
\begin{tikzpicture}[scale=.3]
\tkzInit[xmax=2,ymax=2]\tkzGrid[help lines]
 \tkzDefPoint(0,0){P1}
 \tkzDefPoint(2,1){P2}
 \tkzDefPoint(1,2){P3}
 \tkzDefPoint(1,1){Q1}
 \tkzDrawSegments[color=black](P1,P2 P2,P3 P3,P1)
 \tkzDrawPoints[size=2](P1,P2,P3,Q1)
\end{tikzpicture}
\ \\
$3$ &
\begin{tikzpicture}[scale=.3]
\tkzInit[xmax=3,ymax=3]\tkzGrid[help lines]
 \tkzDefPoint(0,0){P1}
 \tkzDefPoint(3,1){P2}
 \tkzDefPoint(1,3){P3}
 \tkzDefPoint(1,1){Q1}
 \tkzDefPoint(2,2){Q2}
 \tkzDefPoint(2,1){Q3}
 \tkzDefPoint(1,2){Q4}
 \tkzDefPoint(1,2){Q5}
 \tkzDrawSegments[color=black](P1,P2 P2,P3 P3,P1)
 \tkzDrawPoints[size=2](P1,P2,P3,Q1,Q2,Q3,Q4)
\end{tikzpicture}
\hspace{1mm}
\begin{tikzpicture}[scale=.3]
\tkzInit[xmax=3,ymax=3]\tkzGrid[help lines]
 \tkzDefPoint(2,3){P1}
 \tkzDefPoint(0,0){P2}
 \tkzDefPoint(3,1){P3}
 \tkzDefPoint(3,2){P4}
 \tkzDefPoint(1,1){Q1}
 \tkzDefPoint(2,1){Q2}
 \tkzDefPoint(2,2){Q3}
 \tkzDrawSegments[color=black](P1,P2 P2,P3 P3,P4 P4,P1)
 \tkzDrawPoints[size=2](P1,P2,P3,P4,Q1,Q2,Q3)
\end{tikzpicture}
\ \\
$4$ &
\begin{tikzpicture}[scale=.3]
\tkzInit[xmax=4,ymax=4]\tkzGrid[help lines]
 \tkzDefPoint(4,1){P1}
 \tkzDefPoint(2,4){P2}
 \tkzDefPoint(0,0){P3}
 \tkzDefPoint(1,0){P4}
 \tkzDefPoint(2,2){Q1}
 \tkzDefPoint(1,1){Q2}
 \tkzDefPoint(3,2){Q3}
 \tkzDefPoint(2,1){Q4}
 \tkzDefPoint(2,3){Q5}
 \tkzDefPoint(3,1){Q6}
 \tkzDefPoint(1,2){Q7}
 \tkzDrawSegments[color=black](P1,P2 P2,P3 P3,P4 P4,P1)
 \tkzDrawPoints[size=2](P1,P2,P3,P4,Q1,Q2,Q3,Q4,Q5,Q6,Q7)
\end{tikzpicture}
\hspace{1mm}
\begin{tikzpicture}[scale=.3]
\tkzInit[xmax=4,ymax=4]\tkzGrid[help lines]
 \tkzDefPoint(3,4){P1}
 \tkzDefPoint(1,3){P2}
 \tkzDefPoint(0,0){P3}
 \tkzDefPoint(4,2){P4}
 \tkzDefPoint(2,3){Q1}
 \tkzDefPoint(2,2){Q2}
 \tkzDefPoint(2,1){Q3}
 \tkzDefPoint(1,2){Q4}
 \tkzDefPoint(1,1){Q5}
 \tkzDefPoint(3,3){Q6}
 \tkzDefPoint(3,2){Q7}
 \tkzDrawSegments[color=black](P1,P2 P2,P3 P3,P4 P4,P1)
 \tkzDrawPoints[size=2](P1,P2,P3,P4,Q1,Q2,Q3,Q4,Q5,Q6,Q7)
\end{tikzpicture}
\hspace{1mm}
\begin{tikzpicture}[scale=.3]
\tkzInit[xmax=4,ymax=4]\tkzGrid[help lines]
 \tkzDefPoint(3,2){P1}
 \tkzDefPoint(0,0){P2}
 \tkzDefPoint(1,4){P3}
 \tkzDefPoint(2,4){P4}
 \tkzDefPoint(4,3){P5}
 \tkzDefPoint(1,1){Q1}
 \tkzDefPoint(2,3){Q2}
 \tkzDefPoint(2,2){Q3}
 \tkzDefPoint(3,3){Q4}
 \tkzDefPoint(1,2){Q5}
 \tkzDefPoint(1,3){Q6}
 \tkzDrawSegments[color=black](P1,P2 P2,P3 P3,P4 P4,P5 P5,P1)
 \tkzDrawPoints[size=2](P1,P2,P3,P4,P5,Q1,Q2,Q3,Q4,Q5,Q6)
\end{tikzpicture}
\hspace{1mm}
\begin{tikzpicture}[scale=.3]
\tkzInit[xmax=4,ymax=4]\tkzGrid[help lines]
 \tkzDefPoint(0,0){P1}
 \tkzDefPoint(4,1){P2}
 \tkzDefPoint(1,4){P3}
 \tkzDefPoint(3,1){Q1}
 \tkzDefPoint(2,3){Q2}
 \tkzDefPoint(2,2){Q3}
 \tkzDefPoint(2,1){Q4}
 \tkzDefPoint(1,3){Q5}
 \tkzDefPoint(1,2){Q6}
 \tkzDefPoint(1,1){Q7}
 \tkzDefPoint(3,2){Q8}
 \tkzDrawSegments[color=black](P1,P2 P2,P3 P3,P1)
 \tkzDrawPoints[size=2](P1,P2,P3,Q1,Q2,Q3,Q4,Q5,Q6,Q7,Q8)
\end{tikzpicture}
\hspace{1mm}
\begin{tikzpicture}[scale=.3]
\tkzInit[xmax=4,ymax=4]\tkzGrid[help lines]
 \tkzDefPoint(0,0){P1}
 \tkzDefPoint(4,2){P2}
 \tkzDefPoint(3,3){P3}
 \tkzDefPoint(1,4){P4}
 \tkzDefPoint(2,1){Q1}
 \tkzDefPoint(2,2){Q2}
 \tkzDefPoint(2,3){Q3}
 \tkzDefPoint(1,1){Q4}
 \tkzDefPoint(1,2){Q5}
 \tkzDefPoint(1,3){Q6}
 \tkzDefPoint(3,2){Q7}
 \tkzDrawSegments[color=black](P1,P2 P2,P3 P3,P4 P4,P1)
 \tkzDrawPoints[size=2](P1,P2,P3,P4,Q1,Q2,Q3,Q4,Q5,Q6,Q7)
\end{tikzpicture}
\hspace{1mm}
\begin{tikzpicture}[scale=.3]
\tkzInit[xmax=4,ymax=4]\tkzGrid[help lines]
 \tkzDefPoint(0,0){P1}
 \tkzDefPoint(4,3){P2}
 \tkzDefPoint(1,4){P3}
 \tkzDefPoint(0,2){P4}
 \tkzDefPoint(2,2){Q1}
 \tkzDefPoint(0,1){Q2}
 \tkzDefPoint(1,2){Q3}
 \tkzDefPoint(1,1){Q4}
 \tkzDefPoint(1,3){Q5}
 \tkzDefPoint(3,3){Q6}
 \tkzDefPoint(2,3){Q7}
 \tkzDefPoint(0,2){Q8}
 \tkzDefPoint(3,3){Q9}
 \tkzDefPoint(1,2){Q10}
 \tkzDefPoint(1,4){Q11}
 \tkzDrawSegments[color=black](P1,P2 P2,P3 P3,P4 P4,P1)
 \tkzDrawPoints[size=2](P1,P2,P3,P4,Q1,Q2,Q3,Q4,Q5,Q6,Q7)
\end{tikzpicture}
\hspace{1mm}
\begin{tikzpicture}[scale=.3]
\tkzInit[xmax=4,ymax=4]\tkzGrid[help lines]
 \tkzDefPoint(1,3){P1}
 \tkzDefPoint(4,1){P2}
 \tkzDefPoint(2,4){P3}
 \tkzDefPoint(0,0){P4}
 \tkzDefPoint(4,1){Q1}
 \tkzDrawPoints[size=2](Q1)
 \tkzDefPoint(2,2){Q2}
 \tkzDrawPoints[size=2](Q2)
 \tkzDefPoint(2,4){Q3}
 \tkzDrawPoints[size=2](Q3)
 \tkzDefPoint(1,1){Q4}
 \tkzDrawPoints[size=2](Q4)
 \tkzDefPoint(3,2){Q5}
 \tkzDrawPoints[size=2](Q5)
 \tkzDefPoint(1,3){Q6}
 \tkzDrawPoints[size=2](Q6)
 \tkzDefPoint(2,1){Q7}
 \tkzDrawPoints[size=2](Q7)
 \tkzDefPoint(0,0){Q8}
 \tkzDrawPoints[size=2](Q8)
 \tkzDefPoint(2,3){Q9}
 \tkzDrawPoints[size=2](Q9)
 \tkzDefPoint(3,1){Q10}
 \tkzDrawPoints[size=2](Q10)
 \tkzDefPoint(1,2){Q11}
 \tkzDrawPoints[size=2](Q11)
 \tkzDrawSegments[color=black](P1,P4)
 \tkzDrawSegments[color=black](P1,P3)
 \tkzDrawSegments[color=black](P2,P3)
 \tkzDrawSegments[color=black](P2,P4)
\end{tikzpicture}
\end{tabular}
\end{center}
\end{proposition}
\begin{proof}
We use the database of polygons with 
small volume~\cite{bal} to analyze all
the polygons with volume at most
$16$ and such that $\vol(\Delta)
- m^2 \leq 0$.
For each such polygon $\Delta$
we compute $\ls_\Delta(m)$, where 
$m\leq 4$, with the aid of the function
\texttt{FindCurves} of the Magma library:
\begin{center}
\url{https://github.com/alaface/non-polyhedral/blob/master/lib.m}
\end{center}
We take only the pairs $(\Delta,m)$ such that $\ls_\Delta(m)$ contains exactly one element, which is irreducible. Finally we check that in each of these cases the Newton polygon coincides with $\Delta$.
\end{proof}

\begin{remark}
The above intrinsic curves are all expected. In all but the last case
they are intrinsic $(-1)$-curves,
in the last case the curve has 
self-intersection $0$ and genus $1$.
The smallest value of $m$ for an 
unexpected intrinsic negative curve 
is $5$. The lattice polygon $\Delta$ 
is the following

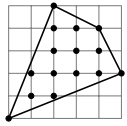
\begin{center}
\begin{tikzpicture}[scale=.3]
\tkzInit[xmax=5,ymax=5]\tkzGrid
 \tkzDefPoint(0,0){P1}
 \tkzDefPoint(2,5){P2}
 \tkzDefPoint(4,4){P3}
 \tkzDefPoint(5,2){P4}
 \tkzDefPoint(1,1){Q1}
 \tkzDefPoint(3,2){Q2}
 \tkzDefPoint(3,4){Q3}
 \tkzDefPoint(2,2){Q4}
 \tkzDefPoint(2,4){Q5}
 \tkzDefPoint(4,3){Q6}
 \tkzDefPoint(1,2){Q7}
 \tkzDefPoint(3,3){Q8}
 \tkzDefPoint(2,1){Q9}
 \tkzDefPoint(2,3){Q10}
 \tkzDefPoint(4,2){Q11}
 \tkzDrawSegments[color=black](P1,P2 P2,P3 P3,P4 P4,P1)
 \tkzDrawPoints[size=2](P1,P2,P3,P4,Q1,Q2,Q3,Q4,Q5,Q6,Q7,Q8,Q9,Q10,Q11)
\end{tikzpicture}
\end{center}
One has $|\partial\Delta\cap\mathbb Z^2| = m-1$ 
and $|\Delta\cap\mathbb Z^2| = \binom{m+1}{2}$
which imply that the corresponding curve 
has arithmetic genus $1$.
The curve is defined by the Laurent polynomial 
{\small
\begin{gather*}
 1 - 8uv + 3uv^2 + 6u^2v^4 - u^2v^5 
 + 3u^2v + 20u^2v^2\\ - 18u^2v^3 
 - 18u^3v^2 + 8u^3v^3 
 + 6u^4v^2  - u^4v^4  - u^5v^2,
\end{gather*}
}
which is the unique one whose Newton polygon
is contained in $\Delta$ and has multiplicity
$5$ at $(1,1)$. Its strict transform in 
$X_\Delta$ is smooth of genus $1$ and 
self-intersection $-1$.
\end{remark}

The proof of Proposition~\ref{pro:class}
suggests the following definitions for
a pair $(\Delta,m)$.
\begin{definition}
Let $\Delta$ be a lattice polygon, $m$ a positive integer, and set 
$p_a := \frac{1}{2}(\vol(\Delta) 
 - |\partial\Delta\cap\mathbb Z^2|
 + m - m^2)+1$.
We say that $(\Delta,m)$ is:
\begin{itemize}
\item
{\em numerically negative} (resp. {\em non positive}) if $\vol(\Delta) - m^2 < 0$ (resp. $\leq$);
\item a $(-n)$-{\em pair}
if $\vol(\Delta) - m^2 = -n < 0$ and $p_a = 0$;
\item {\em expected}
if $|\Delta\cap \mathbb Z^2| >
{{m+1}\choose{2}}$.
\end{itemize}
\end{definition}

\begin{remark}
\label{rem:neg}
Clearly, if $C$ is an intrinsic negative curve, then the pair $(\Delta,m)$, consisting of the Newton polygon of $C$ and
the multiplicity of $\Gamma = \pi(C)$ at $e$, is numerically negative.
On the other hand, if a pair $(\Delta,m)$ is numerically negative, in general there does not exist an intrinsic negative curve
associated to it. Indeed, first of all it can happen that $\ls_\Delta(m)$
is empty (see Example~\ref{ex:empty}).
Furthermore, even if $(\Delta,m)$ is expected (so that $\ls_\Delta(m)$ is not empty), in some cases it contains only reducible curves (see Example~\ref{ex:red}). 
\end{remark}

\begin{example}
\label{ex:empty}
Consider the following polygon $\Delta$ 
\begin{center}
\begin{tikzpicture}[scale=.3]
\tkzInit[xmax=4,ymax=4]\tkzGrid[help lines]
 \tkzDefPoint(0,0){P1}
 \tkzDefPoint(1,4){P2}
 \tkzDefPoint(2,4){P3}
 \tkzDefPoint(4,3){P4}
 \tkzDefPoint(1,1){Q1}
 \tkzDrawPoints[size=2](Q1)
 \tkzDefPoint(2,4){Q2}
 \tkzDrawPoints[size=2](Q2)
 \tkzDefPoint(2,3){Q3}
 \tkzDrawPoints[size=2](Q3)
 \tkzDefPoint(2,2){Q4}
 \tkzDrawPoints[size=2](Q4)
 \tkzDefPoint(3,3){Q5}
 \tkzDrawPoints[size=2](Q5)
 \tkzDefPoint(4,3){Q6}
 \tkzDrawPoints[size=2](Q6)
 \tkzDefPoint(1,4){Q7}
 \tkzDrawPoints[size=2](Q7)
 \tkzDefPoint(0,0){Q8}
 \tkzDrawPoints[size=2](Q8)
 \tkzDefPoint(1,3){Q9}
 \tkzDrawPoints[size=2](Q9)
 \tkzDefPoint(1,2){Q10}
 \tkzDrawPoints[size=2](Q10)
 \tkzDrawSegments[color=black](P1,P2)
 \tkzDrawSegments[color=black](P2,P3)
 \tkzDrawSegments[color=black](P3,P4)
 \tkzDrawSegments[color=black](P1,P4)
\end{tikzpicture}
\end{center}
We have that $\vol(\Delta) = 14$ and $|\partial\Delta\cap\mathbb Z^2| = 4$, 
so that the pair $(\Delta,4)$ is numerically a $(-2)$-pair.
A direct computation shows that $\ls_\Delta(4) = \emptyset$,
so that there does not exists an intrinsic $(-2)$-curve 
with Newton polygon $\Delta$ and multiplicity $4$.
\end{example}

\begin{example}
\label{ex:red}
The polygon $\Delta$, whose Minkowski
decomposition $\Delta = \Delta_1 +
\Delta_2$ is given in the below picture
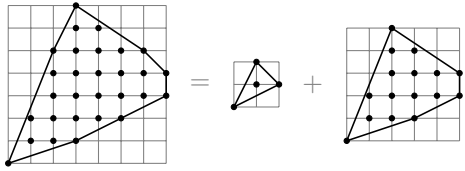
\begin{center}
\begin{tikzpicture}[scale=.3]
\tkzInit[xmax=7,ymax=7]\tkzGrid[help lines]
 \tkzDefPoint(0,0){P1}
 \tkzDefPoint(3,1){P2}
 \tkzDefPoint(7,3){P3}
 \tkzDefPoint(7,4){P4}
 \tkzDefPoint(6,5){P5}
 \tkzDefPoint(3,7){P6}
 \tkzDefPoint(2,5){P7}
 \tkzDefPoint(2,3){Q1}
 \tkzDrawPoints[size=2](Q1)
 \tkzDefPoint(2,2){Q2}
 \tkzDrawPoints[size=2](Q2)
 \tkzDefPoint(2,1){Q3}
 \tkzDrawPoints[size=2](Q3)
 \tkzDefPoint(3,7){Q4}
 \tkzDrawPoints[size=2](Q4)
 \tkzDefPoint(7,4){Q5}
 \tkzDrawPoints[size=2](Q5)
 \tkzDefPoint(5,5){Q6}
 \tkzDrawPoints[size=2](Q6)
 \tkzDefPoint(3,6){Q7}
 \tkzDrawPoints[size=2](Q7)
 \tkzDefPoint(7,3){Q8}
 \tkzDrawPoints[size=2](Q8)
 \tkzDefPoint(5,4){Q9}
 \tkzDrawPoints[size=2](Q9)
 \tkzDefPoint(0,0){Q10}
 \tkzDrawPoints[size=2](Q10)
 \tkzDefPoint(3,5){Q11}
 \tkzDrawPoints[size=2](Q11)
 \tkzDefPoint(5,3){Q12}
 \tkzDrawPoints[size=2](Q12)
 \tkzDefPoint(3,4){Q13}
 \tkzDrawPoints[size=2](Q13)
 \tkzDefPoint(5,2){Q14}
 \tkzDrawPoints[size=2](Q14)
 \tkzDefPoint(3,3){Q15}
 \tkzDrawPoints[size=2](Q15)
 \tkzDefPoint(3,2){Q16}
 \tkzDrawPoints[size=2](Q16)
 \tkzDefPoint(3,1){Q17}
 \tkzDrawPoints[size=2](Q17)
 \tkzDefPoint(1,2){Q18}
 \tkzDrawPoints[size=2](Q18)
 \tkzDefPoint(6,5){Q19}
 \tkzDrawPoints[size=2](Q19)
 \tkzDefPoint(1,1){Q20}
 \tkzDrawPoints[size=2](Q20)
 \tkzDefPoint(4,6){Q21}
 \tkzDrawPoints[size=2](Q21)
 \tkzDefPoint(6,4){Q22}
 \tkzDrawPoints[size=2](Q22)
 \tkzDefPoint(4,5){Q23}
 \tkzDrawPoints[size=2](Q23)
 \tkzDefPoint(6,3){Q24}
 \tkzDrawPoints[size=2](Q24)
 \tkzDefPoint(4,4){Q25}
 \tkzDrawPoints[size=2](Q25)
 \tkzDefPoint(2,5){Q26}
 \tkzDrawPoints[size=2](Q26)
 \tkzDefPoint(4,3){Q27}
 \tkzDrawPoints[size=2](Q27)
 \tkzDefPoint(2,4){Q28}
 \tkzDrawPoints[size=2](Q28)
 \tkzDefPoint(4,2){Q29}
 \tkzDrawPoints[size=2](Q29)
 \tkzDrawSegments[color=black](P1,P7)
 \tkzDrawSegments[color=black](P6,P7)
 \tkzDrawSegments[color=black](P5,P6)
 \tkzDrawSegments[color=black](P4,P5)
 \tkzDrawSegments[color=black](P3,P4)
 \tkzDrawSegments[color=black](P2,P3)
 \tkzDrawSegments[color=black](P1,P2)
 \node at (8.5,3.5) {=};
 \begin{scope}[xshift=10cm,yshift=2.5cm]
\tkzInit[xmax=2,ymax=2]\tkzGrid[help lines]
 \tkzDefPoint(0,0){P1}
 \tkzDefPoint(2,1){P2}
 \tkzDefPoint(1,2){P3}
 \tkzDefPoint(1,1){Q1}
 \tkzDrawPoints[size=2](Q1)
 \tkzDefPoint(2,1){Q2}
 \tkzDrawPoints[size=2](Q2)
 \tkzDefPoint(0,0){Q3}
 \tkzDrawPoints[size=2](Q3)
 \tkzDefPoint(1,2){Q4}
 \tkzDrawPoints[size=2](Q4)
 \tkzDrawSegments[color=black](P1,P3)
 \tkzDrawSegments[color=black](P2,P3)
 \tkzDrawSegments[color=black](P1,P2)
 \node at (3.5,1) {+};
\end{scope}

\begin{scope}[xshift=15cm,yshift=1cm]
\tkzInit[xmax=5,ymax=5]\tkzGrid[help lines]
 \tkzDefPoint(0,0){P1}
 \tkzDefPoint(3,1){P2}
 \tkzDefPoint(5,2){P3}
 \tkzDefPoint(5,3){P4}
 \tkzDefPoint(2,5){P5}
 \tkzDefPoint(2,2){Q1}
 \tkzDrawPoints[size=2](Q1)
 \tkzDefPoint(4,3){Q2}
 \tkzDrawPoints[size=2](Q2)
 \tkzDefPoint(2,4){Q3}
 \tkzDrawPoints[size=2](Q3)
 \tkzDefPoint(1,1){Q4}
 \tkzDrawPoints[size=2](Q4)
 \tkzDefPoint(3,2){Q5}
 \tkzDrawPoints[size=2](Q5)
 \tkzDefPoint(5,3){Q6}
 \tkzDrawPoints[size=2](Q6)
 \tkzDefPoint(3,4){Q7}
 \tkzDrawPoints[size=2](Q7)
 \tkzDefPoint(2,1){Q8}
 \tkzDrawPoints[size=2](Q8)
 \tkzDefPoint(0,0){Q9}
 \tkzDrawPoints[size=2](Q9)
 \tkzDefPoint(4,2){Q10}
 \tkzDrawPoints[size=2](Q10)
 \tkzDefPoint(2,3){Q11}
 \tkzDrawPoints[size=2](Q11)
 \tkzDefPoint(2,5){Q12}
 \tkzDrawPoints[size=2](Q12)
 \tkzDefPoint(3,1){Q13}
 \tkzDrawPoints[size=2](Q13)
 \tkzDefPoint(5,2){Q14}
 \tkzDrawPoints[size=2](Q14)
 \tkzDefPoint(3,3){Q15}
 \tkzDrawPoints[size=2](Q15)
 \tkzDefPoint(1,2){Q16}
 \tkzDrawPoints[size=2](Q16)
 \tkzDrawSegments[color=black](P1,P5)
 \tkzDrawSegments[color=black](P4,P5)
 \tkzDrawSegments[color=black](P3,P4)
 \tkzDrawSegments[color=black](P2,P3)
 \tkzDrawSegments[color=black](P1,P2)
\end{scope}

\end{tikzpicture}
\end{center}
has $\vol(\Delta) = 48$, $|\partial\Delta\cap\mathbb Z^2| = 8$ 
and
$|\Delta\cap\mathbb Z^2| = {{7+1}\choose{2}} + 1$, so that
$(\Delta,7)$ is an expected $(-1)$-pair. The only
element in $\ls_\Delta(7)$ is
defined by
{\small
\begin{gather*}
(u^2 v + u v^2 - 3 u v + 1)\cdot
    (u^5 v^3 - 2 u^5 v^2 - 6 u^4 v^3 + 11 u^4 v^2 - 2 u^3 v^4 + 17 u^3 v^3 -\\ 
        24 u^3 v^2 - u^3 v - u^2 v^5 + 7 u^2 v^4 - 22 u^2 v^3 + 21 u^2 v^2 + 
        5 u^2 v + 4 u v^2 - 9 u v + 1).
\end{gather*}
}
The factorization implies the Minkowski
decomposition $\Delta = \Delta_1 + \Delta_2$. 
The polygon $\Delta_1$ corresponds to an intrinsic $(-1)$-curve $C_1$, while 
$\Delta_2$ corresponds to an intrinsic
curve $C_2$ of self-intersection $0$ and genus $1$.
By Proposition~\ref{pro:int} it follows
that $C_1\cdot C_2 = 0$.
\end{example}

\begin{remark}
\label{rem:neg-int}
Finally we observe that if the pair $(\Delta,m)$ is numerically non positive and $\ls_\Delta(m)$ contains an irreducible curve $\Gamma$, then a Laurent polynomial $f$ of $\Gamma\cap (\mathbb K^*)^2$ defines an intrinsic non positive curve. Indeed, by definition the strict transform $C\subseteq X_\Delta$ of $\Gamma$ satisfies $C^2 = \vol(\Delta) - m^2 \leq 0$. Moreover, since the Newton polygon $\Delta'$ of $f$ is contained in $\Delta$, we also have that $\vol(\Delta') - m^2 \leq \vol(\Delta) - m^2 \leq 0$.
We remark that if $\Delta'$ is strictly contained in $\Delta$, then the self intersection of the intrinsic curve is strictly smaller than $C^2$
(see Example~\ref{ex:-1-2}).

\end{remark}

\begin{example}
\label{ex:-1-2}
Let $\Delta$ and $\Delta'$ be the 
following polygons, respectively from
left to right
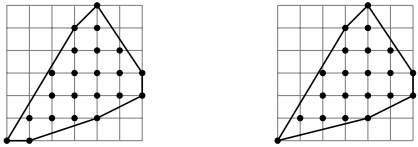
\begin{center}
\begin{tikzpicture}[scale=.3]
\tkzInit[xmax=6,ymax=6]\tkzGrid[help lines]
 \tkzDefPoint(0,0){P1}
 \tkzDefPoint(1,0){P2}
 \tkzDefPoint(4,1){P3}
 \tkzDefPoint(6,2){P4}
 \tkzDefPoint(6,3){P5}
 \tkzDefPoint(4,6){P6}
 \tkzDefPoint(3,5){P7}
 \tkzDefPoint(2,3){Q1}
 \tkzDrawPoints[size=2](Q1)
 \tkzDefPoint(4,1){Q2}
 \tkzDrawPoints[size=2](Q2)
 \tkzDefPoint(2,2){Q3}
 \tkzDrawPoints[size=2](Q3)
 \tkzDefPoint(2,1){Q4}
 \tkzDrawPoints[size=2](Q4)
 \tkzDefPoint(5,4){Q5}
 \tkzDrawPoints[size=2](Q5)
 \tkzDefPoint(0,0){Q6}
 \tkzDrawPoints[size=2](Q6)
 \tkzDefPoint(3,5){Q7}
 \tkzDrawPoints[size=2](Q7)
 \tkzDefPoint(5,3){Q8}
 \tkzDrawPoints[size=2](Q8)
 \tkzDefPoint(3,4){Q9}
 \tkzDrawPoints[size=2](Q9)
 \tkzDefPoint(5,2){Q10}
 \tkzDrawPoints[size=2](Q10)
 \tkzDefPoint(3,3){Q11}
 \tkzDrawPoints[size=2](Q11)
 \tkzDefPoint(3,2){Q12}
 \tkzDrawPoints[size=2](Q12)
 \tkzDefPoint(3,1){Q13}
 \tkzDrawPoints[size=2](Q13)
 \tkzDefPoint(1,1){Q14}
 \tkzDrawPoints[size=2](Q14)
 \tkzDefPoint(4,6){Q15}
 \tkzDrawPoints[size=2](Q15)
 \tkzDefPoint(4,5){Q16}
 \tkzDrawPoints[size=2](Q16)
 \tkzDefPoint(1,0){Q17}
 \tkzDrawPoints[size=2](Q17)
 \tkzDefPoint(6,3){Q18}
 \tkzDrawPoints[size=2](Q18)
 \tkzDefPoint(4,4){Q19}
 \tkzDrawPoints[size=2](Q19)
 \tkzDefPoint(6,2){Q20}
 \tkzDrawPoints[size=2](Q20)
 \tkzDefPoint(4,3){Q21}
 \tkzDrawPoints[size=2](Q21)
 \tkzDefPoint(4,2){Q22}
 \tkzDrawPoints[size=2](Q22)
 \tkzDrawSegments[color=black](P1,P7)
 \tkzDrawSegments[color=black](P6,P7)
 \tkzDrawSegments[color=black](P5,P6)
 \tkzDrawSegments[color=black](P4,P5)
 \tkzDrawSegments[color=black](P3,P4)
 \tkzDrawSegments[color=black](P2,P3)
 \tkzDrawSegments[color=black](P1,P2)

\begin{scope}[xshift=12cm]
\tkzInit[xmax=6,ymax=6]\tkzGrid[help lines]
 \tkzDefPoint(0,0){P1}
 \tkzDefPoint(4,1){P2}
 \tkzDefPoint(6,2){P3}
 \tkzDefPoint(6,3){P4}
 \tkzDefPoint(4,6){P5}
 \tkzDefPoint(3,5){P6}
 \tkzDefPoint(2,3){Q1}
 \tkzDrawPoints[size=2](Q1)
 \tkzDefPoint(4,1){Q2}
 \tkzDrawPoints[size=2](Q2)
 \tkzDefPoint(2,2){Q3}
 \tkzDrawPoints[size=2](Q3)
 \tkzDefPoint(2,1){Q4}
 \tkzDrawPoints[size=2](Q4)
 \tkzDefPoint(5,4){Q5}
 \tkzDrawPoints[size=2](Q5)
 \tkzDefPoint(0,0){Q6}
 \tkzDrawPoints[size=2](Q6)
 \tkzDefPoint(3,5){Q7}
 \tkzDrawPoints[size=2](Q7)
 \tkzDefPoint(5,3){Q8}
 \tkzDrawPoints[size=2](Q8)
 \tkzDefPoint(3,4){Q9}
 \tkzDrawPoints[size=2](Q9)
 \tkzDefPoint(5,2){Q10}
 \tkzDrawPoints[size=2](Q10)
 \tkzDefPoint(3,3){Q11}
 \tkzDrawPoints[size=2](Q11)
 \tkzDefPoint(3,2){Q12}
 \tkzDrawPoints[size=2](Q12)
 \tkzDefPoint(3,1){Q13}
 \tkzDrawPoints[size=2](Q13)
 \tkzDefPoint(1,1){Q14}
 \tkzDrawPoints[size=2](Q14)
 \tkzDefPoint(4,6){Q15}
 \tkzDrawPoints[size=2](Q15)
 \tkzDefPoint(4,5){Q16}
 \tkzDrawPoints[size=2](Q16)
 \tkzDefPoint(6,3){Q17}
 \tkzDrawPoints[size=2](Q17)
 \tkzDefPoint(4,4){Q18}
 \tkzDrawPoints[size=2](Q18)
 \tkzDefPoint(6,2){Q19}
 \tkzDrawPoints[size=2](Q19)
 \tkzDefPoint(4,3){Q20}
 \tkzDrawPoints[size=2](Q20)
 \tkzDefPoint(4,2){Q21}
 \tkzDrawPoints[size=2](Q21)
 \tkzDrawSegments[color=black](P1,P6)
 \tkzDrawSegments[color=black](P5,P6)
 \tkzDrawSegments[color=black](P4,P5)
 \tkzDrawSegments[color=black](P3,P4)
 \tkzDrawSegments[color=black](P2,P3)
 \tkzDrawSegments[color=black](P1,P2)
\end{scope}
\end{tikzpicture}
\end{center}
One has $\vol(\Delta) = 35,\ |\partial\Delta\cap\mathbb Z^2| = 7$ and $|\Delta\cap\mathbb Z^2| = {{6+1}\choose{2}} + 1$, so that $(\Delta,6)$ is an expected $(-1)$-pair.
The linear system $\ls_\Delta(6)$
contains a unique irreducible curve
defined by the following polynomial
{\small
\begin{gather*}
f := -4 {u^6 v^3} + 3 {u^6 v^2} - 
6 u^5 v^4 + 30 u^5 v^3 - 18 u^5 v^2 - {u^4 v^6} + 2 u^4 v^5 + 17 u^4 v^4 - 
62 u^4 v^3 + 25 u^4 v^2\\ + 4 {u^4 v} + 
4 {u^3 v^5} - 26 u^3 v^4 + 50 u^3 v^3 + 
2 u^3 v^2 - 10 u^3 v + 6 u^2 v^3 - 27 u^2 v^2 +  
6 u^2 v + 6 u v - {1}.
\end{gather*}
}
The Newton polygon of $f$ is
$\Delta'$, since $u$ is the only
monomial (corresponding to a lattice
point of $\Delta$) that does not appear in $f$.
In particular, $\vol(\Delta') = 34$ and $|\partial\Delta\cap\mathbb Z^2| = 6$, so 
that $f$ defines an intrinsic (unexpected) $(-2)$-curve.
\end{example}

\begin{definition}
Given two polygons $\Delta_1$,
$\Delta_2$ their {\em mixed volume} is:
\[
 \vol(\Delta_1,\Delta_2) 
 :=
 \frac{1}{2}(\vol(\Delta_1+\Delta_2)
 - \vol(\Delta_1) - \vol(\Delta_2)).
\]
\end{definition}
We conclude the section by showing
how the mixed volume of two lattice
polygons relates with the intersection
product of curves on a toric variety
whose fan refines the normal fans
of the two polygons.

\begin{proposition}
\label{pro:int}
Let $(\Delta_1,m_1)$, $(\Delta_2,m_2)$,
be two pairs, each of which consists of
a lattice polygon together with a
positive integer. Assume that
$\ls_{\Delta_i}(m_i)$ is non-empty
and let $C_i\subseteq X_{\Delta_i}$ be the strict transform of a curve in the
linear system. Let $X$ be a surface which admits birational morphisms
$\phi_i\colon X\to X_{\Delta_i}$ for
$i=1,2$. Then 
\[
 \phi_1^*(C_1)\cdot \phi_2^*(C_2) 
 = 
 \vol(\Delta_1,\Delta_2)-m_1m_2.
\]
\end{proposition}
\begin{proof}
Let $\pi_i\colon X_{\Delta_i}\to
\mathbb P_{\Delta_i}$ be the blowing-up
at $e\in\mathbb P_{\Delta_i}$ 
with exceptional divisor $E_i$.
Since $C_i\sim
\pi_i^*H_i-m_iE_i$, where $H_i$ 
is a very ample divisor on $\mathbb P_{\Delta_i}$, we can
assume that the support of the divisor 
$\pi_i^*H_i-m_iE_i$ does not contain any singular point of $X_{\Delta_i}$. Thus
$C_i$ is a Cartier divisor of $X_{\Delta_i}$ and the pullback $\phi_i^*$
is defined on $C_i$.
The intersection product $\phi_1^*(C_1)\cdot \phi_2^*(C_2)$
does not depend on the surface $X$
because all such surfaces differ by 
exceptional divisors, which have 
zero intersection product with the 
pullbacks of $C_1$ and $C_2$.
We can then choose $X := X_\Delta$ to
be the blowing-up of $\mathbb P_\Delta$ at the general point $e$, where 
$\Delta := \Delta_1+\Delta_2$.
Since $H_i$ is very ample on 
$\mathbb P_{\Delta_i}$, its pullback 
is base point free in $\mathbb P_\Delta$.
By Bertini's theorem the general 
elements of these two linear systems intersect transversely 
at distinct points which, without
loss of generality, we can assume to be contained in $(\mathbb K^*)^2$.
By Bernstein-Kushnirenko theorem 
the number of these intersections
is $\vol(\Delta_1,\Delta_2)$, so that 
the statement follows after taking into
account the intersections of the two 
curves at $e\in\mathbb P_\Delta$.
\end{proof}

\section{Infinite families}
 \label{s:inf}
In this section we construct infinite
families of intrinsic negative curves.
First of all we produce infinite families of toric surfaces, each of which corresponds to an intrinsic negative curve.
Then, in Example~\ref{ex:inf}, we construct 
an infinite family of intrinsic negative curves
on a given toric surface.

\begin{lemma}
\label{lem:mult}
Let $f_1,f_2,f_3,f_4\in\mathbb K[t]$ and let 
$m$ be the maximal degree of the 
four polynomials. Assume that 
$f_1 - f_2 = f_4 - f_3$, 
${\rm Gcd}(f_1,f_2) = {\rm Gcd}(f_3,f_4) =
1$, and $\deg(f_1-f_2) = m$.
Then the image of the following rational map 
\[
 \xymatrix@1{
\pp^1\ar@{..>}[r] & (\mathbb K^*)^2,
 }
 \qquad
 t \mapsto  
 (f_1/f_2,f_3/f_4)
\]
has multiplicity $m$ at $(1,1)$.
\end{lemma}
\begin{proof}
Since $\deg(f_1-f_2) = m$ and $\mathbb K$ algebraically closed, it follows that 
$f_1-f_2$ has $m$ roots.
Any root of $f_1-f_2$ is also a root of $f_3-f_4$, so that
we conclude that there are $m$ values of $t$ (counting multiplicities) whose image is the point $(1,1)$.
\end{proof}

\begin{proof}[Proof of Theorem~\ref{prop:inf}]
First of all, each polygon $\Delta$ of type (v) satisfies $\lw(\Delta) = m$, $\vol(\Delta) = m^2$ and $|\partial\Delta\cap\mathbb Z^2| = m$, so that 
the pair $(\Delta,m)$ is numerically $0$ and expected (in particular it has arithmetic genus $1$). Moreover, in~\cite{cltu}*{Section~6} it has been shown that if we set $m = 2k+4$, for each $k\geq 1$ there exists an irreducible curve of genus $1$ whose Newton polygon is $\Delta$. Therefore we are left with cases (i) to (iv), for which 
the arithmetic genus of the pair is $0$.
In these cases, consider the polynomial functions $f_1,f_2,f_3,f_4 := f_1-f_2+f_3$
given in the following table.

{\small
\begin{center}
\begin{longtable}{|c|c|c|c|}
\hline
& $f_1$  & $f_2$ & $f_3$ \\
\hline
&&&\\
(i) & $-1$ & $\displaystyle{\sum_{i=1}^{m} t^i}$ & 
$t^m$\\
&&&\\
\hline
&&&\\
(ii) & $(m-1)t - (m-2)$ & $-(t-1)t^{m-1} $ & $-(t-1)^3
\displaystyle{\left(t^{m-3} + \sum_{i=0}^{m-4}(m-2-i)t^i\right)}$\\
&&&\\
\hline
&&&\\
(iii) &
$
a^{2k-2}(t-1)$
&
$
\frac{t^{2k-3}\left(t-
a^2\right)
\left(t^2-
a^2\right)}{a^2}$
&
$
\frac{t^{2k-1}\left(t-
a^2\right)}
{a^2},
\quad a := \frac{k-1}{k-2}$\\
&&&\\
\hline
&&&\\
(iv) & $2t-1$ & $(1-t)t^{m-1}$
&
$-(t-1)^2\left(\displaystyle\sum_{i=1}^{m-2}t^i - 1\right)$\\
&&&\\
\hline
\end{longtable}
\end{center}
}

These polynomials satisfy the 
hypotheses of Lemma~\ref{lem:mult}, 
so that the image of the map 
$\varphi(t) = (f_1/f_2,f_3/f_4)$ 
has a point of multiplicity $m$ at $e$. In order to conclude we have to show 
that in each case the Newton polygon is
the one given in the first column of
the table within the proposition.
To this aim we will use \cite{ds}*{Thm.~1.1} which, given
a parametric curve $\Gamma\subseteq
(\mathbb K^*)^2$, 
provides a description of the normal
fan of the Newton polygon of $\Gamma$
together with the length of the edges,
in terms of the zeroes of the four
polynomials $f_1,\dots,f_4$.
For the sake of completeness we explain in detail
case (i). In this case the map $\varphi$ is defined by
\[
\varphi(t) = 
\left(-\frac1{t\sum_{i=0}^{m-1}t^i},
-\frac{t^m}{\sum_{i=0}^{m-1}t^i}
\right).
\]
Since $\varphi$ satisfies $\ord_0(\varphi) = (-1,m),\,
\ord_{\infty}(\varphi) = (m,-1)$ and $\ord_{q_i}(\varphi)
= (-1,-1)$, for all the $m-1$ roots $q_1,\dots,q_{m-1}$ of $\sum_{i=0}^{m-1}t^i$, 
these are the only values of $t$ for which $\ord(\varphi)$ does not vanish.
By~\cite{ds}*{Thm.~1.1}, the rays of the normal fan of the Newton polygon of 
$\overline{\varphi(\pp^1)}$ 
are $(-1,m),\, (m,-1)$ and $(-1,-1)$.
Moreover, the first two rays correspond to two edges of
lattice length $1$ while the third one has length $m-1$. 
We conclude that the Newton polygon 
has vertices $(0,0),\, (m,1),\, (1,m)$.
\end{proof}

\begin{remark}
The  triangles of type (i) in Theorem~\ref{prop:inf} 
are indeed equivalent to the ones with vertices 
$(0,0),\, (m-1,0),\, (m,m+1)$, i.e. $IT(m-1,1)$ in the notation 
of~\cite{ggk}*{Thm.~1.1.A}.
Therefore, as a byproduct of Theorem~\ref{prop:inf} we obtain an alternative (short) proof of~\cite{ggk1}*{Thm.~1.1}.
\end{remark}

Observe that for each infinite family of Theorem~\ref{prop:inf}, the slopes of (some of) the edges change with $m$, so that also the toric surfaces change. We are now going to give an example of an infinite family of negative curves lying on 
the blowing-up at of a fixed toric surface.
First of all we recall a construction 
from~\cite{cltu}. Given an expected 
lattice polygon
$\Delta\subseteq\mathbb Q^2$ of width 
$m := \lw(\Delta)$, with 
$\vol(\Delta) = m^2$ and 
$|\partial\Delta\cap\mathbb Z^2| = m$,
if $\ls_\Delta(m)$ contains an unique
irreducible element, then its strict
transform $C\subseteq X := X_\Delta$ is a 
curve of arithmetic genus one with 
$C^2 = 0$. Whenever $C$ is smooth we denote
by
\[
 \res\colon\Pic(X)\to\Pic(C)
\]
the pullback induced by the inclusion.
It is not difficult to show that the image
of the above map is contained in 
$\Pic(C)(\mathbb Q)$. If
$\res(C)\in\Pic^0(C)$ is non-torsion
then, by~\cite{cltu}*{Sec. 3}, the divisor
$K_X+C$ is linearly equivalent to an
effective divisor whose support can be
contracted by a birational morphism
$\phi\colon X\to Y$.
The surface $Y$ has at most Du Val
singularities and nef anticanonical
divisor $-K_Y\sim C$ (here with abuse
of notation we denote by the same letter
$C$ a curve which lives in different 
birational surfaces and is disjoint from the exceptional locus).

\begin{lemma}
\label{lem:2dim}
If $\Pic(Y)$ has rank three then
$K_Y^\perp\cap \overline{\Eff}(Y)
= \mathbb Q_{\geq 0}\cdot [C]$.
\end{lemma}
\begin{proof}
The class $[C]$ spans an extremal ray of 
$\overline{\Eff}(Y)$ because the curves
contracted by $\phi$ are disjoint from $C$
and thus $\res(C)$ is non-torsion also on
$Y$.
As a consequence $\overline{\Eff}(Y)$ is
non-polyhedral by~\cite{cltu}*{Lem. 3.3}.
By~\cite{cltu}*{Lem. 3.14} the minimal
resolution of singularities 
$\pi\colon Z\to Y$
is a smooth rational surface $Z$ of
Picard rank $10$, nef anticanonical class
$-K_Z$ and non-polyhedral effective cone
$\overline{\Eff}(Z)$.
Observe that the root sublattice of 
$\Pic(Z)$ spanned by classes of $(-2)$-curves
over singularities of $Y$ has rank $R = 7$.
Assume now that $D$ is an
effective divisor such that
$D\cdot K_Y = 0$. Then $D$ is push-forward
of an effective divisor $D'$ of $Z$ with
$D'\cdot K_Z = 0$. Since $-K_Z$ is nef, by 
adjunction $D'=\sum_ia_iC_i + nC$, where 
each $C_i$ is a $(-2)$-curve and $a_i,n\geq 0$. By~\cite{cltu}*{Cor. 3.17}, the fact
that $\res(C_i) = 0$ for any $i$ and the
fact that $\overline{\Eff}(Y)$ is non-polyhedral we conclude that all the
$C_i$ are contracted by $\pi$.
Thus $D$ is linearly equivalent to a 
positive multiple of $C$.
\end{proof}
We are now going to consider a particular lattice polygon $\Delta$ satisfying the above conditions (it is number~24 in~\cite{cltu}*{Table~3}) and we are going to show that the blowing up of the corresponding toric surface contains infinitely many negative curves (see Remark~\ref{rem:neg_k}).
\begin{proposition}
\label{ex:inf}
Let $X$ be the blowing-up at a
general point of the toric surface $\pp$, defined by the following lattice polygon
$\Delta\subseteq\mathbb Q^2$
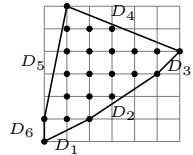
\begin{center}
\begin{tikzpicture}[scale=.3]
\tkzInit[xmax=6,ymax=6]\tkzGrid[help lines]
 \tkzDefPoint(0,0){P1}
 \tkzDefPoint(2,1){P2}
 \tkzDefPoint(5,3){P3}
 \tkzDefPoint(6,4){P4}
 \tkzDefPoint(1,6){P5}
 \tkzDefPoint(0,1){P6}
 \tkzDefPoint(2,3){Q1}
 \tkzDrawPoints[size=2](Q1)
 \tkzDefPoint(2,2){Q2}
 \tkzDrawPoints[size=2](Q2)
 \tkzDefPoint(2,1){Q3}
 \tkzDrawPoints[size=2](Q3)
 \tkzDefPoint(0,1){Q4}
 \tkzDrawPoints[size=2](Q4)
 \tkzDefPoint(5,4){Q5}
 \tkzDrawPoints[size=2](Q5)
 \tkzDefPoint(0,0){Q6}
 \tkzDrawPoints[size=2](Q6)
 \tkzDefPoint(3,5){Q7}
 \tkzDrawPoints[size=2](Q7)
 \tkzDefPoint(1,6){Q8}
 \tkzDrawPoints[size=2](Q8)
 \tkzDefPoint(5,3){Q9}
 \tkzDrawPoints[size=2](Q9)
 \tkzDefPoint(3,4){Q10}
 \tkzDrawPoints[size=2](Q10)
 \tkzDefPoint(1,5){Q11}
 \tkzDrawPoints[size=2](Q11)
 \tkzDefPoint(3,3){Q12}
 \tkzDrawPoints[size=2](Q12)
 \tkzDefPoint(1,4){Q13}
 \tkzDrawPoints[size=2](Q13)
 \tkzDefPoint(3,2){Q14}
 \tkzDrawPoints[size=2](Q14)
 \tkzDefPoint(1,3){Q15}
 \tkzDrawPoints[size=2](Q15)
 \tkzDefPoint(1,2){Q16}
 \tkzDrawPoints[size=2](Q16)
 \tkzDefPoint(1,1){Q17}
 \tkzDrawPoints[size=2](Q17)
 \tkzDefPoint(6,4){Q18}
 \tkzDrawPoints[size=2](Q18)
 \tkzDefPoint(4,4){Q19}
 \tkzDrawPoints[size=2](Q19)
 \tkzDefPoint(2,5){Q20}
 \tkzDrawPoints[size=2](Q20)
 \tkzDefPoint(4,3){Q21}
 \tkzDrawPoints[size=2](Q21)
 \tkzDefPoint(2,4){Q22}
 \tkzDrawPoints[size=2](Q22)
 \tkzDrawSegments[color=black](P1,P6)
 \tkzDrawSegments[color=black](P5,P6)
 \tkzDrawSegments[color=black](P4,P5)
 \tkzDrawSegments[color=black](P3,P4)
 \tkzDrawSegments[color=black](P2,P3)
 \tkzDrawSegments[color=black](P1,P2)
 \tkzLabelSegment[below](P1,P2){\tiny $D_1$}
 \tkzLabelSegment[below](P2,P3){\tiny $D_2$}
 \tkzLabelSegment[below, pos=1](P3,P4){\tiny $D_3$}
 \tkzLabelSegment[above](P4,P5){\tiny $D_4$}
 \tkzLabelSegment[left](P5,P6){\tiny $D_5$}
 \tkzLabelSegment[left](P6,P1){\tiny $D_6$}
\end{tikzpicture}
\end{center}
Then there is a birational morphism 
$\phi\colon X\to Y$ onto a rational surface
$Y$ of Picard rank three with only Du Val 
singularities. Moreover if $D_1,\dots,D_6$ 
are the pullbacks of the prime invariant 
divisors of $\pp$, ordered according to the
picture, the pushforward on $Y$ of the
divisor

{\small
\begin{gather*}
 E_k :=
 (7k^2-1)D_3 
 + (161k^2 - 49k - 9)D_4 
 + (70k^2 - 53k + 9)D_5 \\
 + (14k^2 - 12k + 2)D_6 
 - (42k^2 - 19k)E
\end{gather*}
}
is linearly equivalent to a $(-1)$-curve
for any integer $k\neq 0$.
\end{proposition}
\begin{proof}
Let $C$ be the curve of $X$ 
defined by the unique element
in $\ls_\Delta(6)$, which 
has equation
{\small
\begin{gather*}
 f := - {1}+2  {v}+7 u v
 -3  {u^2 v}-23 u v^2+6 u^2 v^2+2 u^3 v^2+18 u v^3+20 u^2 v^3
 -26 u^3 v^3
 +10 u^4 v^3-2\\  {u^5 v^3}
 -12 u v^4-11 u^2 v^4+6 u^3 v^4+5 u^4 v^4-4
 u^5 v^4+ {u^6 v^4}+5 u v^5
 +3 u^2 v^5-2 u^3 v^5- {u v^6}.
\end{gather*}
}
Let us denote by $v_1,\dots,v_6$ the 
primitive generators of the rays of 
the normal fan to $\Delta$, which are
the columns of the following matrix
\[
P:=
 \begin{pmatrix*}[r]
 -1 &-2 &-1 & -2 & 5 & 1\\
  2 & 3 & 1 & -5 & -1 & 0
 \end{pmatrix*}.
\]
By~\eqref{very ample} the
divisor $\pi(C)$ is linearly equivalent to 
$-\sum_{i=1}^6\min_{w\in \Delta}
 \{w\cdot v_i\} \pi(D_i)$ so that
\[
 C\sim D_2 + 2D_3 + 32D_4 + D_5 - 6E.
\]
By~\eqref{equ:num} the curve $C$ has 
self-intersection $C^2=0$ and it is
smooth of genus $1$.
It is isomorphic, over the 
rational numbers, to the elliptic curve
of equation $y^2 + y = x^3 + x^2$,
labelled \href{https://www.lmfdb.org/EllipticCurve/Q/43/a/1}{43.a1} in the LMFDB database.
Its Mordell-Weil group $\Pic^0(C)(\mathbb Q)$ 
is free of rank one so that $\res(C)$ is
either trivial or non-torsion. The first
possibility is ruled out by the fact that
$\dim |C| = \dim\ls_\Delta(6) = 0$ and 
the exact sequence of the ideal sheaf
of $C$ in $X$
(see~\cite{cltu}*{Lem. 3.2} for details).
Thus $\res(C)$ is non-torsion, which implies that $h^0(X,nC) = 1$, for any 
positive integer $n$, so that
$[C]$ spans an extremal ray of 
$\overline{\Eff}(X)$.
By~\cite{cltu}*{Cor. 3.12} the
divisor $K_X+C$ is linearly
equivalent to an effective
divisor whose support can be contracted.
This contraction is the morphism $\phi$
in the statement.
We claim that
\[
 K_X+C\sim 3C_1+2C_2,
\]
where $C_1\sim 5D_4+D_5-E$
and $C_2\sim 5D_5+D_6-E$ are the strict transforms of the two one-parameter
subgroups corresponding to the width directions $(1,0)$ and 
$(0,1)$ of $\Delta$.
To prove the claim it suffices to observe
that the divisor
$K_X+C - 3C_1-2C_2\sim 
-D_1 + D_3 + 16D_4 -13D_5 - 3D_6$
is principal, being a linear
combination of the rows of the 
above matrix $P$.
Using the intersection matrix of 
$D_1,\dots,D_6,E$
\[
 \begin{pmatrix}
 -3/2  &   1 & 0 & 0 & 0 & 1/2 &  0\\
    1  & -1 & 1 & 0 & 0  &  0 & 0\\
    0  &  1 & -16/7 & 1/7 & 0  & 0 & 0\\
    0  &  0 & 1/7 & 4/189 & 1/27 & 0 & 0\\
    0  &  0 &  0  & 1/27 & -5/27  & 1 & 0\\
  1/2  &  0 &  0 & 0  &  1 & -9/2  &  0\\
    0  &  0 &  0  &  0 & 0 & 0 &  -1
\end{pmatrix}
\]
we can see that $C_i\cdot C = 
C_1\cdot C_2 = 0$ for $i=1,2$
and $E_k$ has integer intersection product with all the $D_i$, in particular it is a Cartier divisor. Moreover,
\begin{equation}
\label{equs}
 E_k\cdot C_1 = E_k\cdot C_2 = 0
 \qquad
 E_k^2 = E_k\cdot K = -1,
\end{equation}
so that by Riemann-Roch $E_k$ is effective.
To prove that $\phi_*E_k$ is irreducible
we proceed as follows.
Since $X$ has Picard rank $5$ and 
$\phi\colon X\to Y$ contracts 
$C_1\cup C_2$, the surface $Y$ 
has Picard rank $3$.
The push-forward $\phi_*E_k$ is an effective and Cartier divisor, because we are contracting curves which have intersection product zero with $E_k$. Moreover, since 
$-K_Y$ is nef and
\[
 -K_Y\cdot\phi_*E_k = 1,
\]
we deduce that 
$\phi_*E_k$ is either irreducible or it can be written as $D + D'$, with $D$ irreducible and reduced and $D'$ orthogonal to $K_Y$. By Lemma~\ref{lem:2dim},
$D'$ must be equivalent to a multiple of $C$, so that we can write $\phi_*E_k = D+nC$, for some $n>0$.
Since $D$ is Cartier, both
$D^2$ and $D\cdot K_Y$ are integers,
and moreover, being $-K_Y$ nef, by the genus formula $D^2\geq -2$.
The case $D^2=-2$ can be ruled out
since otherwise $D$ would be in $K_Y^\perp$.
Thus $D^2 = D\cdot K_Y = -1$ and 
$-1 = \phi_*E_k^2 = D^2+2nD\cdot 
C = -1 +2n > 0$ gives again a contradiction.
It follows that $\phi_*E_k$ is linearly 
equivalent to a $(-1)$-curve.
\end{proof}

\begin{remark}
\label{rem:neg_k}
The way we determined the divisors $E_k$
has been by solving the diophantine 
equations~\ref{equs}. 
We also remark that
a priori the curve $E_k$ could be reducible
of the form $E_k = \Gamma_k + a_1(k)C_1+a_2(k)C_2$, 
with $\Gamma_k$ irreducible and $a_1(k),a_2(k)\geq 0$.
So we are only showing the existence of 
infinitely many negative curves $\Gamma_k$
on $X$, which are not necessarily $(-1)$-curves.
Moreover even if $E_k = \Gamma_k$,
the Newton polygon of $E_k$ does not necessarily
coincide with the Riemann-Roch polygon
$\Delta_k$ of the curve $E_k$. So that $E_k$
could be a $(-1)$-curve but not an intrinsic 
one.
The Riemann-Roch polygon has vertices
corresponding to the columns of the following
matrix
\[
 \begin{pmatrix}
 0 & 0 & 7k^2-4k & 14k^2-12k+2 & 35k^2-12k-1 & 42k^2-19k\\
 0 & 7k^2+k & 42k^2 -19k & 7k^2-6k+1 & 21k^2-6k-1 & 28k^2-13k
 \end{pmatrix}.
\]
In particular, if we set $m = 42k^2-19k$, we have that 
$\vol(\Delta_k) = m^2 - 1$, $|\partial\Delta_k\cap \mathbb Z^2| = m +1$
and $\lw(\Delta_k) = m$, so that $(\Delta_k,m)$ is numerically a $(-1)$-pair.
For small values of $k$ it is possible to check, with 
the help of the software Magma~\cite{mag}, that the Newton polygon of
the $(-1)$-curve is indeed $\Delta_k$, but for general 
$k$ we are not able to prove it. Therefore 
$\Gamma_k$ is not necessarily an intrinsic $(-1)$-curve, but it is anyway an intrinsic negative curve (see Remark~\ref{rem:neg}).
\end{remark}

\section{Seshadri constants}
 \label{s:sesh}
In this section we first prove Theorem~\ref{prop:sesh-m} and Corollary~\ref{cor:inf}, and then we discuss some consequences on the study of the effective cone of the blowing up of weighted projective planes.

We will need the following preliminary result about Seshadri constants on projective surfaces.
\begin{lemma}
\label{lem:sesh}
Let $Y$ be a projective surface, 
$H$ an ample divisor of $Y$, and
let $\pi\colon X\to Y$ be the blowing-up 
of $Y$ at a smooth point $p\in Y$ with exceptional
divisor $E$. 
\begin{enumerate}
\item
\label{lem:1}
If there is a positive 
integer $m$ such that $\pi^*H-mE$ is
the class of an effective curve 
$C = \sum_{i=1}^ra_iC_i$ with $C^2\leq 0$,
then
\[
 \varepsilon(H,p) 
 =
 \min_i\left\{\frac{\pi^*H\cdot C_i}{E\cdot C_i}\right\}
 \leq m + \frac{C^2}{m}.
\]
\item
\label{lem:2}
If furthermore $C$ is
irreducible, then $\varepsilon(H,p)
 = m + \frac{C^2}{m}$.
\end{enumerate}
\end{lemma}
\begin{proof}
We prove~\ref{lem:1}.
Let $\varepsilon := \varepsilon(H,p)$.
Observe that we can write
$C^2 + (m-\varepsilon)E \cdot C
= (\pi^*H-mE + (m-\varepsilon)E)\cdot C 
= (\pi^*H-\varepsilon E)\cdot C \geq 0$, 
and, since $E\cdot C = m$,
we get
\[
 \varepsilon
 \leq m+\frac{C^2}{m}
 \leq m.
\]
If $C$ is nef then $\varepsilon\geq m$, 
so that $\varepsilon = m$ and $C^2 = 0$.
This implies $C\cdot C_i = 0$, and hence 
$\varepsilon = 
\pi^*H\cdot C_i/E\cdot C_i$ for any $i$,
proving the statement in this case.
If $C$ is not nef then $\varepsilon < m$.
If $\alpha$ is such that $\varepsilon< \alpha < m$ 
then $\pi^*H-\alpha E$ is effective and non-nef.
Let $C'$ be an irreducible curve such that
$(\pi^*H-\alpha E)\cdot C' < 0$, then 
$(\pi^*H-mE)\cdot C' < 0$ as well.
Thus $C' = C_i$ for some $i$.
Since $\alpha$ can be chosen arbitrarily
close to $\varepsilon$, we conclude that
$(\pi^*H-\varepsilon E)\cdot C_j = 0$
for some $j$, and the statement follows.
Statement~\ref{lem:2} is an immediate
consequence of~\ref{lem:1} and the equality
$\pi^*H\cdot C = C^2+m^2$.
\end{proof}

\begin{proof}[Proof of Theorem~\ref{prop:sesh-m}]
We prove~\ref{pro2-i}.
Let $v\in N$ be a width direction,
that is $\lw_v(\Delta) = \lw(\Delta)$
and let $C_v\subseteq X_\Delta$ 
be the strict transform
of the one-parameter subgroup of
the torus defined by $v$. 
If $\mu > \lw(\Delta)$, then $(\pi^*H - \mu E)\cdot C_v
< 0$, so that $\pi^*H - \mu E$ 
is not nef.
This proves the statement.

We prove~\ref{pro2-ii}.
Observe that 
$(\pi^*H-\varepsilon E)^2\geq
(\pi^*H-\lw(\Delta) E)^2 
= \vol(\Delta) - \lw(\Delta)^2 > 0$,
where the first inequality is by~\ref{pro2-i}.
By the Riemann-Roch theorem,
the class of $\pi^*H-\varepsilon E$ 
is in the interior of the effective 
cone $\Eff(X_\Delta)$.
It follows that the Seshadri constant
is computed by a curve $C\subseteq
X_\Delta$. From 
$(\pi^*H-\varepsilon E)\cdot C = 0$
one concludes that $\varepsilon$
is a rational number (see also~\cite{klm}*{Rem. 2.3}).

Statements~\ref{pro2-iii} and
\ref{pro2-iv} are consequence of
Lemma~\ref{lem:sesh} and the fact
that $C^2 = \vol(\Delta) - m^2$.
\end{proof}

\begin{proof}[Proof of Corollary~\ref{cor:inf}]
Observe that by hypothesis 
$\vol(\Delta) - m^2 = C^2 \leq 0$,
so that the hypothesis~\ref{pro2-iii}
of Theorem~\ref{prop:sesh-m} is satisfied.
Moreover, being $C$ irreducible we
have $\lw(\Delta)-m = C\cdot C_v\geq 0$,
where $C_v$ is the strict transform 
of the one-parameter subgroup of the
torus defined by the width direction
$v$. Thus also hypothesis~\ref{pro2-iv}
of Theorem~\ref{prop:sesh-m} is satisfied and
the statement follows.
\end{proof}

\begin{remark}
\label{rem:ito}
Observe that the best lower bound for $\varepsilon$ we can get from~\cite{ito}*{Thm.~1.3} in the case of a toric surface is either the width $\lw(\Delta)$, or the biggest length of a segment inside $\Delta$.
For instance, if $\Delta$ is a triangle of type (i) in Theorem~\ref{prop:inf}, consider the projection $\pi\colon
\qq^2\to \qq$ onto the second factor.
If we take the fiber $\Delta\cap \pi^{-1}(1)$, by~\cite{ito}*{Thm.~1.3} we have the inequality
\[
 \varepsilon(H,e) \geq \min\{m,m-1/m\} = m-1/m.
\]
Since $\vol(\Delta) = m^2-1$, by Theorem~\ref{prop:sesh-m}~\ref{pro2-iii},
we also have the inequality $\varepsilon(H,e) \leq (m^2-1)/m$,
so that we can conclude that $\varepsilon(H,e)$ is indeed equal
to $m-1/m$, no need of showing that there exists an irreducible
curve $C\in\ls_\Delta(m)$.
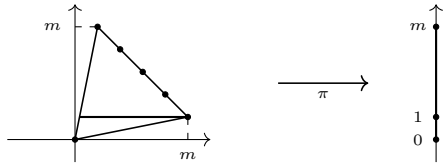
\begin{center}
\begin{tikzpicture}[scale=.3]
\tkzInit[xmax=4,ymax=4]
 \tkzDefPoint(0,0){P1}
 \tkzDefPoint(5,1){P2}
 \tkzDefPoint(5,0){PP2}
 \tkzDefPoint(1,5){P3}
 \tkzDefPoint(0,5){PP3}
 \tkzDefPoint(3,3){Q1}
 \tkzDefPoint(0.2,1){Q}
 \tkzDefPoint(16,-1){R1}
 \tkzDefPoint(16,6){R2}
 \tkzDefPoint(16,0){R3}
 \tkzDefPoint(16,5){R4}
 \tkzDefPoint(0,-1){A1}
 \tkzDefPoint(0,6){A2}
 \tkzDefPoint(-3,0){A3}
 \tkzDefPoint(6,0){A4}
 \tkzDefPoint(16,0){B1}
 \tkzDefPoint(16,5){B2}
 \tkzDefPoint(16,1){O}
 
 \tkzDefPoint(9,2.5){C1}
 \tkzDefPoint(13,2.5){C2}
 
 \tkzDrawPoints[size=2](B1)
 \tkzDrawPoints[size=2](B2)
 \tkzDrawPoints[size=2](O)

 \tkzDrawPoints[size=2](Q1)
 \tkzDefPoint(0,0){Q2}
 \tkzDrawPoints[size=2](Q2)
 \tkzDefPoint(1,5){Q3}
 \tkzDrawPoints[size=2](Q3)
 \tkzDefPoint(4,2){Q4}
 \tkzDrawPoints[size=2](Q4)
 \tkzDefPoint(2,4){Q5}
 \tkzDrawPoints[size=2](Q5)
 \tkzDefPoint(5,1){Q6}
 \tkzDrawPoints[size=2](Q6) 
 \tkzDrawSegments[color=black](P1,P3)
 \tkzDrawSegments[color=black](P2,P3)
 \tkzDrawSegments[color=black](P1,P2)
 
 \filldraw[black] (5,-0.7) node {\tiny{$m$}};
 \filldraw[black] (-1,5) node {\tiny{$m$}};
 \filldraw[black] (15.2,0) node {\tiny{$0$}};
 \filldraw[black] (15.2,1) node {\tiny{$1$}};

 \filldraw[black] (15.2,5) node {\tiny{$m$}};

 \filldraw[black] (11,2)  node {\tiny{$\pi$}};

 \tkzDrawSegments[thick,color=black](Q,P2) 

 \tkzDrawSegments[thick,color=black](R3,R4)

 \tkzDrawSegments[thin,color=black,dashed](P3,PP3) 
 \tkzDrawSegments[thin,color=black,dashed](P2,PP2) 
 \tkzDrawSegments[->,thin,color=black](R1,R2) 
 \tkzDrawSegments[->,thin,color=black](A1,A2) 
 \tkzDrawSegments[->,thin,color=black](A3,A4) 

 \tkzDrawSegments[->,color=black](C1,C2) 
 
\end{tikzpicture}

\end{center}
We also remark that in the remaining cases of Theorem~\ref{prop:inf} 
the bound given by~\cite{ito}*{Thm.~1.3} is not sharp.

In the same vein, if the lattice polygon $\Delta$ contains a segment of lattice length $\lw(\Delta)$, then~\cite{ito}*{Thm.~1.3} gives the bound  $\varepsilon(H,e) \geq \lw(\Delta)$. But since by 
Theorem~\ref{prop:sesh-m}
\ref{pro2-i}
we also have 
the opposite inequality,
we can immediately conclude that the 
Seshadri constant $\varepsilon(H,e)$ is indeed equal to $\lw(\Delta)$. 

For instance, this shows that for any $m\geq 4$, the polygon
with vertices
$(0,0),\, (0,1),\,\\ (m,1),\, (1,m)$
corresponds to a toric surface with Seshadri constant 
$\varepsilon(H,e) = m$ (even if it is not hard to find a
parametrisation as we did with the families of Theorem~\ref{prop:inf}).
\begin{center}
\begin{tikzpicture}[scale=.3]
\tkzInit[xmax=5,ymax=5]\tkzGrid[help lines]
 \tkzDefPoint(0,0){P1}
 \tkzDefPoint(5,1){P2}
 \tkzDefPoint(1,5){P3}
 \tkzDefPoint(0,1){P4}
 \tkzDefPoint(2,3){Q1}
 \tkzDrawPoints[size=2](Q1)
 \tkzDefPoint(4,1){Q2}
 \tkzDrawPoints[size=2](Q2)
 \tkzDefPoint(2,2){Q3}
 \tkzDrawPoints[size=2](Q3)
 \tkzDefPoint(2,1){Q4}
 \tkzDrawPoints[size=2](Q4)
 \tkzDefPoint(0,1){Q5}
 \tkzDrawPoints[size=2](Q5)
 \tkzDefPoint(0,0){Q6}
 \tkzDrawPoints[size=2](Q6)
 \tkzDefPoint(1,5){Q7}
 \tkzDrawPoints[size=2](Q7)
 \tkzDefPoint(3,3){Q8}
 \tkzDrawPoints[size=2](Q8)
 \tkzDefPoint(1,4){Q9}
 \tkzDrawPoints[size=2](Q9)
 \tkzDefPoint(5,1){Q10}
 \tkzDrawPoints[size=2](Q10)
 \tkzDefPoint(3,2){Q11}
 \tkzDrawPoints[size=2](Q11)
 \tkzDefPoint(1,3){Q12}
 \tkzDrawPoints[size=2](Q12)
 \tkzDefPoint(3,1){Q13}
 \tkzDrawPoints[size=2](Q13)
 \tkzDefPoint(1,2){Q14}
 \tkzDrawPoints[size=2](Q14)
 \tkzDefPoint(1,1){Q15}
 \tkzDrawPoints[size=2](Q15)
 \tkzDefPoint(2,4){Q16}
 \tkzDrawPoints[size=2](Q16)
 \tkzDefPoint(4,2){Q17}
 \tkzDrawPoints[size=2](Q17)
 \tkzDrawSegments[color=black](P1,P4)
 \tkzDrawSegments[color=black](P3,P4)
 \tkzDrawSegments[color=black](P2,P3)
 \tkzDrawSegments[color=black](P1,P2)
 \tkzDrawSegments[color=black](P2,P4)
\end{tikzpicture}
\end{center}
\end{remark}

\subsection{Weighted projective planes}
We briefly recall an open problem 
about the existence of certain 
irreducible curves in weighted projective planes, and its 
relation with intrinsic curves.
For a comprehensive reference on known facts 
and open problems on blow-ups of weighted 
projective planes see~\cite{ca}*{Sec. 6}.

Let $a,b,c$ be three positive pairwise 
coprime integers, let $\mathbb P(a,b,c)$
be the corresponding weighted projective plane and let $\pi\colon
X(a,b,c)\to\mathbb P(a,b,c)$ 
be the the blowing-up at the general
point $e := (1,1,1)$
with exceptional divisor $E$.
The divisor class group of $X := 
X(a,b,c)$ is free of rank $2$ and
the effective cone $\Eff(X)$ is in
general unknown.
By the Riemann-Roch theorem, $\Eff(X)$
contains the positive light cone $Q$
(shaded region) 
with extremal rays generated by 
$R_{\pm} = \pi^*H\pm\frac{1}{\sqrt{abc}}E$.
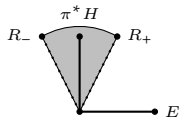
\begin{center}
\begin{tikzpicture}[scale=1]
\tkzDefPoint(0,0){O}
\tkzDefPoint(0,1){H}
\tkzDefPoint(1,0){E}
\tkzDefPoint(-0.5,1){R1}
\tkzDefPoint(0.5,1){R2}
\filldraw[fill=gray!50] (0,0) -- (.5,1) arc
(60:120:1cm) -- cycle;
\tkzDrawPoints[size=2](O,H,E,R1,R2)
\tkzDrawSegments[thick](O,H O,E)
\tkzDrawSegments[thick,dotted](O,R1 O,R2)
\tkzLabelPoint[right](E){\tiny $E$}
\tkzLabelPoint[above](0,1.1){\tiny $\pi^*H$}
\tkzLabelPoint[left](R1){\tiny $R_-$}
\tkzLabelPoint[right](R2){\tiny $R_+$}
\end{tikzpicture}
\end{center}
The question is whether $\Eff(X)$ is bounded by the $\mathbb R$-divisor
$R_-$, so that $\varepsilon(H,e) = 1/\sqrt{abc}$, or by the class of a negative curve  (lying below the ray $R_-$). 
In many examples (see for instance~\cite{ggk} and~\cite{hkl}) the existence of the negative curve has been proved, but in general the question is still open, and in fact it is conjectured that for some triples $a,b,c$ (such as $9,10,13$) that negative curve does not exist.
We remark that proving this conjecture 
would not only give an example of a surface having non rational Seshadri constant, but it would also imply the following well known conjecture 
(see~\cite{nag}), in some particular cases, as explained in Proposition~\ref{prop:nag}.

\begin{conjecture}[Nagata's Conjecture]
Let $\pi\colon X_r\to \mathbb P^2$ be the blowing-up at $r\geq 10$ points in very general position and let $E_1,\dots,E_r$ be the exceptional divisors. Then the class
$\pi^*\mathcal O(1) - \frac1{\sqrt{r}}\sum_i E_i$ is nef on
$X_r$.
\end{conjecture}
So far, Nagata's conjecture have been proved only when $r$ is a perfect square (\cite{nag}), but the following result shows that finding a triple $a,b,c$, such that the Seshadri constant at the general point $e\in X(a,b,c)$ is $1/\sqrt{abc}$, would imply Nagata's conjecture for $r = abc$ (it can be found for instance in~\cite{ck}*{Prop. 5.2}, but we give anyway a brief proof for the sake of completeness).
\begin{proposition}
\label{prop:nag}
If $\varepsilon(H,e) = 1/\sqrt{abc}$
then the Nagata's conjecture holds
for $abc$ points in the plane.
\end{proposition}
\begin{proof}
Let $f\colon\mathbb P^2\to\mathbb P(a,b,c)$ be the morphism defined by 
$(x,y,z)\mapsto (x^a,y^b,z^c)$
and let $Y_r$ be the blowing-up of
$\mathbb P^2$ at the $r := abc$ points
of $f^{-1}(e)$.
Since $R_-$ is nef, also
\[
 f^*R_- =
 L-\frac{1}{\sqrt{abc}}
 \sum_{i=1}^{abc}E_i
\]
is nef on $Y_r$. If we denote by $X_r$ 
the blowing-up of $\mathbb P^2$ at
$r$ points in very general position
then $\overline{\rm Eff}(X_r)
\subseteq\overline{\rm Eff}(Y_r)$
by semicontinuity of the dimension of
cohomology.
Thus $\overline{\rm Nef}(X_r)
\supseteq\overline{\rm Nef}(Y_r)$
so that $f^*R_-$ is nef also on $X_r$.
\end{proof}

\begin{remark}
If $C$ is a very general smooth 
irreducible curve of positive 
genus $g$, then the N\'eron-Severi 
group, over the rational numbers, 
of the symmetric product ${\rm Sym}^2(C)$ 
has rank two.
In~\cite{ciko}*{Prop.~3.1} the authors show
that if the Nagata's conjecture is true and
$g\geq 9$, then the effective cone of 
${\rm Sym}^2(C)$ is open on one side.
\end{remark}

The equality $\varepsilon(H,e) = 1/\sqrt{abc}$ holds if and only if there does not exist a negative curve in $X(a,b,c)$ having class $d\pi^*H- mE$, with $d/m < \sqrt{abc}$.
A partial result in this direction is given by ~\cite{km}*{Thm. 5.4}, where the authors show that if the
negative curve is expected,
then $d$ is bounded from above and
this result allows them to conclude
that there are no such 
curves on certain $X(a,b,c)$, like
e.g. $X(9,10,13)$.

On the other hand we can also say that the equality
$\varepsilon(H,e) = 1/\sqrt{abc}$
holds if and only if there exists
a sequence $\pi^*d_nH- m_nE$ of classes of positive irreducible
curves in $X(a,b,c)$ such that 
$d_n/m_n\to \sqrt{abc}$, that is
these classes approach the ray $R_-$
from the inside of the light cone.
In this direction observe that an 
intrinsic negative curve can appear
as a positive curve in $X(a,b,c)$. 
We discuss this approach 
for $X(9,10,13)$ by producing many
intrinsic $(-1)$-curves which
are positive curves on the surface.
We proceed using the fact that
the Cox ring of $X(a,b,c)$ is 
isomorphic to the extended saturated Rees algebra
(see~\cite{hkl} and
\cite{adhl}*{Prop. 4.1.3.8}):
\[
 R[I]^{\rm sat}
 :=
 \bigoplus_{m\in\mathbb Z}
 (I^m : J^\infty) t^{-m}
 \subseteq R[t^{\pm 1}],
\]
where $R = \mathbb K[x,y,z]$ is the
Cox ring of the weighted projective
plane, $I$ is the ideal of $(1,1,1)$
in Cox coordinates, $J = \langle
x,y,z\rangle$ is the irrelevant
ideal and $I^{-m} = R$ for any
$m\geq 0$. 
Using this we can compute a
minimal generating set consisting of
homogeneous elements of given
bounded multiplicity at $(1,1,1)$. 
In the case of $X(9,10,13)$,
fixing the maximum of the  
multiplicity to be $30$, we found
$52$ generators. In the following
table we display the degrees of
these generators together with the
self-intersection of the corresponding 
intrinsic curve and its genus (while 
the self-intersection of the curve 
on $X(9,10,13)$ is $d^2/abc - m^2$).

\vspace{3mm}

{\tiny
\begin{center}
\begin{tabular}{c|c|c|c}
\hline
$d$ & $m$ & $C^2$ & $p_a$\\
\hline
36 & 1 & 0 & 0\\
39 & 1 & 0 & 0\\
40 & 1 & 0 & 0\\
83 & 2 & -1 & 0\\
109 & 3 & -1 & 0\\
110 & 3 & -1 & 0\\
113 & 3 & -1 & 0\\
139 & 4 & -1 & 0\\
140 & 4 & -1 & 0\\
143 & 4 & -1 & 0\\
208 & 6 & -1 & 0\\
209 & 6 & -1 & 0\\
210 & 6 & -1 & 0\\
\end{tabular}
\quad
\begin{tabular}{c|c|c|c}
\hline
$d$ & $m$ & $C^2$ & $p_a$\\
\hline
213 & 6 & -2 & 0\\
243 & 7 & 0 & 1\\
309 & 9 & -1 & 0\\
310 & 9 & -1 & 0\\
312 & 9 & -1 & 0\\
313 & 9 & -1 & 0\\
378 & 11 & -1 & 0\\
379 & 11 & -1 & 0\\
380 & 11 & -1 & 0\\
413 & 12 & -1 & 0\\
481 & 14 & -1 & 0\\
482 & 14 & -1 & 0\\
483 & 14 & -1 & 0\\
\end{tabular}
\quad
\begin{tabular}{c|c|c|c}
\hline
$d$ & $m$ & $C^2$ & $p_a$\\
\hline
516 & 15 & 0 & 1\\
549 & 16 & -1 & 0\\
550 & 16 & -1 & 0\\
551 & 16 & -1 & 0\\
585 & 17 & -1 & 0\\
652 & 19 & -1 & 0\\
653 & 19 & -1 & 0\\
686 & 20 & 0 & 1\\
720 & 21 & 1 & 2\\
721 & 21 & -1 & 0\\
755 & 22 & -1 & 0\\
789 & 23 & 0 & 1\\
790 & 23 & 1 & 2\\
\end{tabular}
\quad
\begin{tabular}{c|c|c|c}
\hline
$d$ & $m$ & $C^2$ & $p_a$\\
\hline
823 & 24 & -1 & 0\\
824 & 24 & -1 & 0\\
858 & 25 & -1 & 0\\
\textcolor{red}{891} & 
\textcolor{red}{26} & 
\textcolor{red}{-1} & 
\textcolor{red}{0}\\
892 & 26 & -1 & 0\\
893 & 26 & 0 & 1\\
893 & 26 & 3 & 3\\
926 & 27 & -1 & 0\\
\textcolor{red}{959} & 
\textcolor{red}{28} & 
\textcolor{red}{0} & 
\textcolor{red}{1}\\
960 & 28 & 0 & 1\\
994 & 29 & -1 & 0\\
1028 & 30 & 0 & 1\\
1029 & 30 & -1 & 0\\
\end{tabular}
\end{center}
}

\ \\

The slope $d/m$ which best
approximates $\sqrt{9\cdot 10
\cdot 13} \sim 34.20526$
is $959/28 = 34.25$, realized by 
the last curve. The best approximation given by an intrinsic $(-1)$-curve of the above list is $891/26 \sim 34.26923$. The following question naturally arises.
\begin{question}
Is it possible to construct an infinite family of intrinsic $(-1)$-curves appearing as positive curves in $X(9,10,13)$, and whose slopes approach $\sqrt{9\cdot 10\cdot 13}$?
\end{question}

\end{document}